\documentclass[11pt,reqno]{amsart}
\usepackage{graphicx,amscd}
\usepackage{amsfonts,amsmath,amssymb}
\newcommand{\rl}{{\mathbb{R}}}
\newcommand{\cx}{{\mathbb{C}}}

\newcommand{\om}{\omega}

\newcommand{\dbar}{\overline{\partial}}

\newcommand{\e}{\varepsilon}

\newcommand{\dist}{{\mathrm{dist}}}

\newcommand{\B}{\mathbb{B}}

\newtheorem{theorem}{Theorem}[section]
\newtheorem{lemma}[theorem]{Lemma}
\newtheorem{prop}[theorem]{Proposition}
\newtheorem{cor}[theorem]{Corollary}

\theoremstyle{definition}

\begin{document}
\title{Rational approximation and Lagrangian inclusions}
\author{Rasul Shafikov* and Alexandre Sukhov**}
\begin{abstract}
We show that any real compact surface $S$, except the sphere $S^2$ and the projective plane $\mathbb RP_2$, admits 
a pair of smooth complex-valued functions $f_1$, $f_2$ with the property that any continuous
complex-valued function on $S$ is a uniform limit of a sequence of $R_j(f_1,f_2)$, where
$R_j(z_1,z_2)$ are rational functions on $\mathbb C^2$. 
\end{abstract}

\maketitle

\let\thefootnote\relax

* Department of Mathematics, the University of Western Ontario, London, Ontario, N6A 5B7, Canada,
e-mail: shafikov@uwo.ca. The author is partially supported by the Natural Sciences and Engineering 
Research Council of Canada.

**Universit\'e des Sciences et Technologies de Lille, 
U.F.R. de Math\'ematiques, 59655 Villeneuve d'Ascq, Cedex, France,
e-mail: sukhov@math.univ-lille1.fr. The author is partially supported by Labex CEMPI.


\section{Introduction}

This work concerns approximation of continuous functions on a compact real surface 
by a special class of smooth functions. To illustrate this we consider the one-dimensional example first. 
In the space of continuous complex-valued functions on the unit 
circle $S^1\subset \cx$ let $\mathcal R\subset C^0(S^1)$  be the subalgebra of functions of the form $R(e^{i\theta})$, 
where $\theta\in [0,2\pi]$ and $R(z)$ is a rational function on $\cx$ with poles off $S^1$.  It follows from the 
Stone-Weierstrass theorem that $\mathcal R$ is dense in $C^0(S^1)$.  Note that by the maximum principle  
the subspace of  polynomials in $e^{i\theta}$ is  not dense in $C^0(S^1)$.  We consider the case of 
dimension 2. Our main result is the following

\begin{theorem}\label{t.2}
Let $S$ be a smooth compact real surface without boundary, and let $C^0(S)$ be the space
of continuous complex-valued functions on $S$. There exists a pair of smooth functions 
$f_j : S \to \cx$, $j=1,2$, such that for every function $F \in C^0(S)$ there is a sequence 
$\{R_n(z_1,z_2)\}$ of rational functions on $\cx^2$ with the following properties:
\begin{itemize}
\item[(i)] For every $n$  the denominator of the composition $R_n(f_1,f_2)$ does not  vanish on $S$.
\item[(ii)] If $S$ is not the unit sphere $S^2$ and is not the projective plane $\rl P_2$, then
$\{R_n(f_1,f_2)\}$ converges to $F$ in $C^0(S)$.
\item[(iii)] If $S = S^2$, then there exists a rotation $\tau$ of $S^2$ (depending on $F$) such that $\{R_n(f_1,f_2)\}$ converges to the composition $F \circ \tau$ in $C^0(S^2)$.
\item[(iv)] If $S = \rl P_2$, then there exists a smooth diffeomorphism  $\tau$ of $\rl P_2$ (depending on $F$) such that $\{R_n(f_1,f_2)\}$ converges to the composition $F \circ \tau$ in $C^0(\rl P_2)$.
\end{itemize}
\end{theorem}

This result provides an affirmative   answer to the question communicated to us by Nemirovski. 
Note that the pair $f_1, f_2$ is independent of $F$, and that rational functions in Theorem~\ref{t.2} cannot be replaced by
polynomials. To see this, suppose that for a given surface $S$ there exist continuous functions $f_1, f_2$ 
such that any continuous function on $S$ can be approximated by polynomials in $f_1$
and $f_2$. Since $C^0(S)$ separates points on $S$, the map $f=(f_1,f_2) : S\to f(S) \subset \cx^2$ is a bijection, hence a homeomorphism. By assumption,  any continuous function on $f(S)$ can be approximated by holomorphic
polynomials,  which forces $f(S)$ to be polynomially convex in $\cx^2$. Recall that a compact set $X\subset \mathbb C^2$
is {\it polynomially convex} if for every point $z \in \cx^2 \setminus X$ there is a polynomial $P$ such that 
$\vert P(z) \vert > \sup_{w \in X} \vert P(w) \vert$. However, no compact topological $n$-dimensional submanifold of 
$\cx^n$ is polynomially convex, see \cite[Cor. 2.3.5]{St}), and this proves the claim.

The functions $f_1$, and $f_2$ in Theorem~\ref{t.2} will be given as the coordinate components of a singular 
Lagrangian (with respect to the standard symplectic form $\omega_{st}$) embedding of $S$ into $\cx^2$. 
For example, in the simplest case of the torus $S^1\times S^1$, we can take $f_j = e^{i\theta_j}$, $j=1,2$, thinking 
of $\theta_j \in [0,2\pi]$ as a parametrization of each circle $S^1$. For an arbitrary surface we employ in 
Section~2 a  result of 
Givental~\cite{Giv} (see also Audin \cite{A}), who proved the existence on $S$ of a {\it Lagrangian inclusion}--a local Lagrangian embedding 
of $S$  into $\cx^2$ that can have, in addition to transverse double self-intersection points, singularities that are 
called {\it open Whitney umbrellas}; furthermore, such a map is a homeomorphism near every umbrella.
Moreover, one can find such an inclusion without self-intersection points, i.e., a topological embedding, with two exceptions, 
the sphere $S^2$ and the projective plane $\rl P_2$. These two surfaces  do not admit a singular Lagrangian embedding into $\cx^2$, but can be included  
with transverse double points, and so one needs  more functions to generate $C^0(S)$. 

Although no embedding of $S$ into $\cx^2$ is polynomially convex, we prove in Section 3 that there exists a Lagrangian inclusion of $S$ into $\cx^2$  such that its image is  rationally convex. A compact set $X$ in $\cx^n$ is called 
{\it rationally convex} if for every point $z \in \cx^n \setminus X$ there exists a complex algebraic hypersurface passing 
through $z$ and avoiding~$X$.  This is used in the proof of  Theorem \ref{t.2} which is given in  Section 4.

That rational convexity is closely connected with the property of being Lagrangian became apparent from
the work of Duval \cite{D}. Duval and Sibony~\cite{DS}  showed that a compact $n$-dimensional submanifold of $\cx^n$ is rationally 
convex whenever it is Lagrangian with respect to some K\"ahler form. It was further proved by Gayet~\cite{G} that 
an immersed Lagrangian submanifold in $\cx^n$ with transverse double self-intersections is also rationally convex. 
This was generalized to certain nontransverse self-intersections by Duval and Gayet~\cite{DG}.  Interaction between 
Lagrangian  geometry and rational convexity was recently explored by Cieliebak-Eliashberg 
\cite{CE} and Nemirovski-Siegel \cite{NS} using topological methods.

\medskip

\noindent{\bf Acknowledgment.} The authors would like to thank Stefan Nemirovski for numerous
fruitful discussions, in particular for pointing out the connection between the results concerning rational convexity of 
Lagrangian submanifolds and the approximation theory.

\section{Lagrangian embeddings and inclusions}

A nondegenerate closed 2-form $\omega$ on $\cx^2$ is called a {\it symplectic form}. 
By Darboux's theorem every symplectic form is locally equivalent to the standard form
$$
\omega_{\rm st} = \frac{i}{2} (dz \wedge d\bar z + dw \wedge d\bar w) = dd^c\, \phi_{\rm st}, \ \ 
\phi_{\rm st}= |z|^2 + |w|^2,
$$
where $(z,w)$, $z = x + iy$, $w = u + iv$, are complex coordinates in $\cx^2$, and 
$d^c =i(\overline\partial -\partial)$.  If a symplectic form $\om$ is of bidegree $(1,1)$ and strictly 
positive, it is called a {\it K\"ahler form}. A smooth function $\phi$ is called strictly plurisubharmonic 
if $dd^c\,\phi$ is strictly positive definite. It is called a potential of $\om$ if $dd^c \phi = \om$.
A real $n$-dimensional submanifold $S\subset \cx^n$ is called {\it Lagrangian} with respect to 
$\om$ if $\omega|_S=0$. 

It follows from Arnold~\cite{Ar2} that a compact Lagrangian submanifold of $\mathbb C^n$ has
zero Euler characteristic. On the other hand, according to the result of Givental~\cite{Giv},
any compact surface admits a {\it Lagrangian inclusion} into $\cx^2$ (we 
use the terminology introduced in Arnold~\cite{Ar}), i.e., a smooth map $\iota: S \to \cx^2$ which is a 
local Lagrangian embedding (i.e., $\iota^*\omega_{st} = 0$)  except a finite set of singular points that are either transverse  double 
self-intersections (or simply {\it double points}) or the so-called {\it open Whitney umbrellas}. The {\it standard open Whitney umbrella} is the map 
\begin{eqnarray}\label{st-umb}
\pi: \rl^2_{(t,s)} \ni (t,s) \mapsto \left(ts,\frac{2t^3}{3},t^2,s\right) \in \rl^{4}_{(x,u,y,v)} .
\end{eqnarray}
Images of  the standard open Whitney umbrella under complex affine maps
that preserve the symplectic form $\om_{\rm st}$ will also be called standard umbrellas.
Finally, open Whitney umbrellas are defined as images of the standard umbrella under a local 
symplectomorphism, i.e., a local diffeomorphism that preserves the form $\omega_{\rm st}$. 
If $S$ is orientable then each  inclusion satisfies the following topological identity
\begin{equation}\label{e.giv.or}
-\chi(S) + 2 \cdot d - m =0,
\end{equation}
and if $S$ is nonorientable, then 
\begin{equation}\label{e.giv.nor}
\chi(S) + 2 \cdot d - m =0\ \ {\rm mod\ } 4.
\end{equation}
Here $\chi(S)$ is the Euler characteristic of $S$, $d$ is the number of double points, and $m$ is the 
number of umbrella points.
 
In the orientable case, a double point should be counted taking into account its index, 
which comes from some orientation on $S$ and the standard orientation on $\cx^2$. 
In fact, according to the result of Audin~\cite{A}, any combination of 
numbers $\chi(S)$, $d$, and $m$, for which formula~\eqref{e.giv.or} is valid, can be realized 
in a Lagrangian inclusion. In particular, if $\chi(S) \le 0$, then we may choose $d=0$,
and $m=-\chi(S)$. This means that any orientable surface, except the sphere $S^2$,
admits a singular Lagrangian embedding (i.e., inclusion without double points), while 
the Whitney sphere $\mathcal W|_{S^2}: S^2 \to \cx^2$, where 
\begin{equation}\label{e.wh}
\mathcal W: \rl^3 \ni (t, s, \tau) \to (t+i t \tau, s + i s \tau),
\end{equation}
is a Lagrangian immersion of $S^2$ with one double point.

In the nonorientable case formula~\eqref{e.giv.nor} is valid mod 4 according to~\cite{A}. 
Givental~\cite{Giv} showed that if $\chi(S)\le -2$, then in fact we may take $d=0$,
that is, all such surfaces admit a singular Lagrangian embedding into $\cx^2$. He also
gave an explicit construction of a Lagrangian inclusion of $\rl P_2$ with two double 
points and one umbrella.  Recently Nemirovski and Siegel~\cite{NS} gave all possibilities for the 
number of umbrella points that may appear in a singular Lagrangian embedding of an
arbitrary compact surface $S$. These are given by 
\begin{itemize}
\item[(i)] $m = -\chi(S)$ and $\chi \neq 2$, if $S$ is orientable;
\item[(ii)] $(\chi(S),m) \neq (1,1)$ or $(0,0)$, and $m \in 
\{4-3\chi, -3\chi, -3\chi - 4,...,\chi + 4 - 4\lfloor\chi/4+1\rfloor\}$, 
if $S$ is nonorientable.
\end{itemize}
In particular, all nonorientable surfaces except $\rl P_2$ admit a singular Lagrangian 
embedding, while Givental's inclusion of $\rl P_2$ into $\cx^2$ with two double points
and one umbrella has the simplest possible combination of singularities.

Suppose now that $\iota: S \to \cx^2$ is a Lagrangian inclusion with umbrella points
$p_1, \dots, p_m$. Then, in a neighbourhood $U_j$ of every $p_j$, there 
exists a symplectomorphism $\phi_j : U_0 \to U_j$ from a neighbourhood of the origin
in $\cx^2$ that maps the standard umbrella \eqref{st-umb} to $\iota(S)\cap U_j$. 
Any symplectomorphism $\phi$ is locally Hamiltonian. This means that in a 
(simply connected) neighbourhood $U$ there exists a smooth function $h: U \to \mathbb R$,
called the {\it Hamiltonian}, such that the vector field $V_h$, uniquely defined by the equation 
\begin{equation}\label{e.ham}
i(V_h)\,\omega_{st} = d h,
\end{equation}
gives the flow $\phi^h_\tau$ on $U$ with the property that $\phi^h_1=\phi$. Here $i(V_h)$ is the contraction
operator. Conversely, a smooth function $h: \cx^2 \to \rl$ with compact support defines uniquely a vector field
$V_h$ that satisfies~\eqref{e.ham}. The flow of $V_h$ generates a one parameter family of symplectomorphisms
of~$\cx^2$. These symplectomorphisms are the identity outside the support of $h$. 

Let $L_j$ be the linear translation in $\cx^2$ sending $p_j$ to the origin, and let $h_j$ be the Hamiltonian 
of the symplectic maps $L_j^{-1} \circ\phi_j^{-1}$ defined in a neighbourhood 
$U_j$ of $p_j$. Let $h$ be a smooth function on $\cx^2$ that agrees with 
$h_j$ in $U_j$ and vanishes outside a small neighbourhood $\tilde U_j$ of $\overline U_j$.
Then the diffeomorphism $\Phi$ defined by the flow $\phi^h_1$ is a symplectomorphism of $\cx^2$ which is the 
identity map outside $\tilde U_j$. By construction, $\Phi \circ \iota$ is a standard open Whitney umbrella near $p_j$.
Repeating this procedure for all umbrella points gives a new Lagrangian inclusion (denoted again by $\iota$) with only standard umbrellas. Thus we obtain the following version of Givental's theorem.

\begin{prop}\label{p.stand}
Let $S$ be a compact real surface without boundary. There exists  a Lagrangian inclusion $\iota: S \to \cx^2$  such that all its open Whitney umbrella points are standard. Furthermore, if $S\ne S^2$ or $\mathbb RP_2$, then $S$ admits a singular Lagrangian embedding with only standard umbrellas and without double points.
\end{prop}

\section{Rational Convexity of Lagrangian inclusions}
Here we prove the following 

\begin{prop}\label{t.1}
Let $S$ be a compact real surface without boundary and let   $\iota: S \mapsto (\mathbb C^2, \omega_{\rm st})$
be a Lagrangian
inclusion given by Proposition~\ref{p.stand}. Then  $\iota(S)$ is rationally convex in $\mathbb C^2$.
\end{prop}
Proposition \ref{t.1} was already proved by the authors~\cite{SS3} in the special case of a Lagrangian inclusion 
with a single umbrella. We include here a detailed presentation for convenience of the reader.

We will identify $S$ and $\iota(S)$ as a slight abuse of notation. The ball of radius $\e$ centred 
at a point $p$  is denoted by $\B(p,\e)$, and  the standard Euclidean distance between a point $p\in \cx^n$ and a set $Y\subset \cx^n$ is denoted by $\dist(p,Y)$. Our approach is a modification of the method of Duval-Sibony and Gayet. The main tool here is the following result. 

\begin{lemma}[\cite{DS}, \cite{G}] \label{l.g}
Let $\phi$ be a plurisubharmonic $C^\infty$-smooth function on $\cx^n$, and let $h$ be a $C^\infty$-smooth function on $\cx^n$. Let $X= \{|h|=e^\phi \}$ be compact. Suppose that

\begin{itemize}
	\item[(1)] $|h| \le e^\phi$;
	\item[(2)] $\dbar h = O (\dist(\cdot, S)^{\frac{3n+5}{2}})$;
	\item[(3)] $|h| = e^\phi$ with order at least 1 on $S$; 
	\item[(4)] For any point $p\in X$ at least one of the following conditions holds: (i) $h$ is holomorphic in a neighbourhood 
	of $p$, or (ii) $p$ is a smooth point of $S$, and $\phi$ is strictly plurisubharmonic at $p$.
\end{itemize}

Then $X$ is rationally convex. 
\end{lemma}

We remark that if follows from the proof of the lemma in~\cite{G} that in fact, we may assume
that $\phi$ is merely continuous at points where $h$ is holomorphic.

The proof of Proposition~\ref{t.1} consists of finding the functions $\phi$ and $h$ that satisfy 
Lemma~\ref{l.g} with $X=S$. This will be achieved in three steps: we first construct a closed 
$(1,1)$-form 
$\om$ that vanishes near singular points of $S$ and such that $\om|_S=0$. The form $\om$ is a modification of the standard symplectic form $\om_{\rm st}$ in $\cx^2$ near singular points of $S$. Near self-intersection points this is done in the paper of Gayet~\cite{G}, and so we will deal with the umbrella points. Secondly, from $\om$ and its potential $\phi$ we construct the required function $h$. In the last step 
we replace $\phi$ with a function $\phi + \rho$, for a suitable $\rho$, so that the pair $\{\phi+\rho, h\}$ 
satisfies all the conditions of Lemma~\ref{l.g}. 
\medskip

\noindent {\it Step 1: the form $\omega$}. 
Our modification of the form $\omega_{st}$ and its potential is an inductive procedure on the
umbrella points. Let $p_1, \dots, p_m$ be the umbrella points on $S$, $p_j =(x_j,u_j,y_j,v_j)$.
By the assumption in Theorem~\ref{t.1}, after a translation of $p_j$ to the origin, the surface $S$ 
is parametrized near $p_j$ by the mapping $\pi$ given by~\eqref{st-umb}. Let $L_j : (z,w) \to (z,w) -p_j$ 
be the translation of $p_j$ to the origin, so that $ \pi_j = L_j^{-1} \circ \pi$ parametrizes $S$ near $p_j$.

For a function $f$ we have $d^c f = -f_y dx + f_x dy - f_v du + f_u dv$. Using this we have 
$\pi^*d^c \phi_{\rm st} = -2t^2 s dt - \frac{2}{3}t^3 ds$. Consider the pluriharmonic function
$\zeta_1 = \frac{v^2}{2} - \frac{u^2}{2}$.
Then $\pi^*d^c \zeta_1 = \pi^*d^c \phi_{\rm st}$. The function $\phi_{\rm st} - \zeta_1$ 
is strictly plurisubharmonic and satisfies 
\begin{equation}\label{e.*}
\pi^*d^c (\phi_{\rm st} - \zeta_1)=0.
\end{equation}
Let $\phi_1 = (\phi_{st} - \zeta_1) \circ L_1$. Since $L_j$ are $\cx$-linear, they commute with $d^c$.
Therefore, $d^c\phi_1 \vert S = 0$ near $p_1$ and $dd^c \phi_1 = \omega_{st}$. 
Let $r: \rl^{+} \to \rl^+$ be a smooth increasing convex function such that $r(t)=0$
when $t\le\e_1$ and $r(t) = t -c$ when $t>\e_2$, for some suitably chosen $c>0$ and $0<\e_1<\e_2$. 
We choose $\e_2>0$ so small that the set $\{\phi_1<\e_2\}$ does not contain any singular points
of $S$ except $p_1$. Let
\begin{equation}\label{e.ome}
\tilde \phi_1 = r\circ \phi_1, \ \  \om_1 = dd^c(\tilde\phi_1).
\end{equation}
Then $\pi^*\omega_1 =0$ by~\eqref{e.*}. Therefore, the surface $S$ remains Lagrangian with respect to the 
form $\omega_1$. This gives us the required modification of $\om_{\rm st}$ near $p_1$. Note that our 
construction gives two neighbourhoods $U_1\Subset U'_1$ of $p_1$, which can be chosen arbitrarily small, 
so that $\om_1|_{U_1}=0$ and $\omega_1 =\omega_{\rm st}$ in $\cx^2 \setminus U'_1$. On the other hand, the potential $\tilde\phi_1$ is a global modification of $\phi_{st}$ but it remains plurisubharmonic on 
$\cx^2$. 

Consider now the modification of $\tilde\phi_1$ and $\omega_1$ near $p_2$. Up to an additive constant the potential $\tilde\phi_1$ for $\omega_1$ near $p_2$ agrees with $(\phi_{st} - \zeta_1) \circ L_1$. We construct 
$\phi_2$ in the form
$$
\phi_2 = (\tilde\phi_1 - \zeta_2) \circ L' + C,
$$
with a suitable choice of a function $\zeta_2$ and a constant $C$. 
The condition $\pi_2^*d^c \phi_2 =0$ is equivalent to 
$$
\pi^*d^c \left((\phi_st - \zeta_1)\circ L_1 - \zeta_2 \right) = 0 .
$$
This can be achieved by choosing
$$
\zeta_2 = -2x_1x - 2y_1y - v_1v - 3u_1u .
$$
Then $d^c\phi_2 \vert S = 0$ near $p_2$. Further, $\phi_2(p_2) = 0$ by a suitable choice of the constant $C$, 
and $dd^c\phi_2 = \omega_1$. Now take $\tilde\phi_2 = r \circ \phi_2$, where $r$ is as above, and set
$\omega_2 = dd^c \tilde\phi_2$. This gives the required modification near $p_2$. 

This procedure can be repeated for all other $p_j$, $j=2,\dots,m$. Note that at each step the modification
of the function $\tilde \phi_{j-1}$ is obtained by adding linear terms in $(x,u,y,v)$ precomposed with a 
translation. This ensures that the form $\omega_j$ remains unchanged in the complement of some 
small neighbourhood $U'_j$ of the point $p_j$. For the same reason, the function $\tilde\phi_j$ remains
globally plurisubharmonic, which is, in fact, strictly plurisubharmonic outside the union of the neighbourhoods 
$U'_j$.  We repeat this procedure $m$ times for all umbrella points to obtain the function $\tilde \phi$ and
the form $\tilde\omega$.

Denote by $p_{m+1}, \dots, p_N$ the double points of $S$. Then \cite[Prop. 1]{G} gives further modification of the form $\tilde\omega$ and its potential $\tilde\phi$ near the double points. 
Combining everything together yields the following result.

\begin{lemma}\label{l.omega}
Given $\e>0$ sufficiently small, there exists a $(1,1)$-form $\tilde \om$ and $0<\e'<\e$ 
such that 
\begin{itemize}
	\item[(i)] $\tilde \om|_S = 0$;
	\item[(ii)] $\tilde \om = \om$ on 
	$\cx^2 \setminus \left(\cup_{j=1}^N \B(p_j, \e)\right)$.
	\item[(iii)] $\tilde \om$ vanishes on  
	$\mathbb B(p_j,\e')$, $j=1,\dots,N$.
	
\end{itemize}
Furthermore, there exists a smooth function $\tilde \phi$ on $\cx^2$ such that $dd^c \tilde \phi = \tilde \omega$.
The function $\tilde \phi$ is plurisubharmonic on $\mathbb C^2$, and strictly plurisubharmonic on 
$\cx^2 \setminus \left(\cup_{j=1}^N \B(p_j, \e)\right)$.
\end{lemma}

\medskip

\noindent{\it Step 2: the function $h$.}
Let $\iota: S \to \cx^2$ be a Lagrangian inclusion, and $\tilde\phi$ be the potential of the form $\tilde\omega$ 
given by Lemma~\ref{l.omega}. For simplicity we drop tilde from the notation.
We recall the construction in \cite{DS} and \cite{G} of a smooth function $h$ on $\cx^2$ such that 
$|h|\left|_S = e^\phi \right.$ and $\dbar h(z) = O(\dist(z,S)^6)$.   

Let $\tilde S$ be a deformation retract of $S$. Note that it exists because near an umbrella 
point the surface $S$ is the graph of a continuous vector-function. Let $\gamma_k$, 
$k=1,\dots,l$, be the basis in $H_1(\tilde S, \mathbb Z) \cong H_1(S, \mathbb Z)$ supported on 
$S$. Using de Rham's theorem one can find closed forms $\beta_k$ on $\tilde S$ such that 
$\int_{\gamma_\nu}\beta_k = \delta_{\nu k}$, and such that $\beta_k$ vanish in the balls 
$\B(p_j,\e)$ as in Lemma~\ref{l.omega} around the singularities of $S$. 
Further, there exist smooth functions $\psi_k$ with compact support in $\tilde S$ such that 
$\psi_k$ vanish on $S\cup(\cup_{j=1}^N \B(p_j,\e))$, and for $k=1,\dots,l$,
\begin{equation}\label{e.ttt}
\iota^* d^c \phi_k = \iota^* \beta_k .
\end{equation}
Indeed, for each $k$, we set $\phi_k= A(z,w) r_1 + B(z,w) r_2$, where $r_{1}(z,w)$
and $r_{1}(z,w)$ are local defining functions of $S$ and $A, B$ are some unknown functions. 
Plugging this expression into~\eqref{e.ttt} gives a linear system for the restrictions of $A$ and 
$B$ to $S$ that can be solved. A suitable extension of this solution with support in $\tilde S$ 
gives the result. Note that near singular points the extension is identically
zero.

For $\lambda_k>0$ the function $\phi+ \sum_{j=1}^l \lambda_k \psi_k$ agrees
with $\phi$ on $S$. For sufficiently small $\lambda_k$ it is strictly plurisubharmonic 
outside the balls $\B(p_j,\e)$ and globally plurisubharmonic since the functions $\psi_k$
vanish in $\B(p_j,\e)$. Further, there exists a choice of $\lambda_k$ and $M>0$ 
such that for the function 
\begin{equation}\label{e.phi-til}
\tilde \phi = M \left(\phi+ \sum_{j=1}^l \lambda_j \psi_j \right)
\end{equation}
the form $\iota^* d^c\tilde\phi$ is closed on $S$ and has periods which are multiples of $2\pi$. Then
there exists a $C^\infty$-smooth function $\mu: S \to \rl/2\pi\mathbb Z$ that vanishes on the intersection
of $S$ with $\B(p_j,\e)$, $j=1,\dots,N$, and such that $\iota^* d^c \tilde \phi = d\mu$. By~\cite{HW}, there exists a 
function $h$ defined on $\cx^2$ such that 
$$
h|_S = e^{\tilde \phi + i\mu}|_S
$$
and $\dbar h(z) = O(\dist(z,S)^6)$. It follows that $\tilde \phi - \log|h|$ vanishes to order $1$ on $S$.
Note that $h$ is constant near singular points of $S$. Finally, the function $h$ can be suitably 
extended to $\cx^2$ preserving the inequality given by (1) in Lemma~\ref{l.g}.

\medskip

\noindent{\it Step 3: the function $\phi$.}    A closed subset $K$ in $\cx^n$ is called {\it locally polynomially convex} near a point $p \in K$ if for every sufficiently small $\varepsilon > 0$ the intersection $K \cap \overline{\B(p,\varepsilon)}$ is polynomially convex in $\cx^n$.
Again, for simplicity of notation we denote by $\phi$ the function~\eqref{e.phi-til} constructed in 
Step 2. It does not yet satisfy the conditions of Lemma~\ref{l.g}
because there are still some smooth points on $S$ where the function $h$ is not holomorphic
and $\phi$ is not strictly plurisubharmonic. For this we will replace $\phi$ by a 
function $\tilde \phi = \phi + c\cdot\rho$, where the function $\rho$ will be constructed using local 
polynomial convexity of $S$, and $c>0$ will be a suitable constant.

We recall our result from~\cite{SS1,SS2}.

\begin{lemma}\label{l.FS} 
Let $S$ be a Lagrangian inclusion in $\cx^2$, and let $p_0,\dots, p_N$ be its singular points. 
Suppose that $S$ is locally polynomially convex near every singular point. Then there exists a neighbourhood 
$\Omega$ of $S$ in $\cx^2$ and a continuous non-negative plurisubharmonic function $\rho$  on $\Omega$ 
such that $S\cap \Omega = \{ p \in \Omega: \rho(p) = 0 \}$. Furthermore, for every $\delta > 0$ one can 
choose $\rho = (\dist(z,S))^2$ on $\Omega \setminus \cup_{j=1}^N \B(p_j,\delta)$; in particular, it is 
smooth and strictly plurisubharmonic there.
\end{lemma}

The standard open Whitney umbrella is  locally polynomially convex by \cite{SS1}, and $S$ is locally polynomially
convex near transverse double self-intersection points by~\cite{SS2}. For the proof of the lemma
we refer the reader to~\cite{SS2}.

To complete the construction of the function $\phi$, we choose the function $\rho$ in Lemma~\ref{l.FS} 
with $\delta>0$ so small that the balls $\mathbb B(p_j, \delta)$ are contained in balls $\mathbb B(p_j, \e'/2)$ 
given by Lemma~\ref{l.omega}. Note that $\rho$ is defined only in a 
neighbourhood $\Omega$ of $S$, but we can extend it as a smooth function with compact support in $\cx^2$. 
Consider now the function 
$$
\tilde \phi = \phi + c\cdot \rho .
$$
We choose the constant $c>0$ so small that the
function $\tilde \phi$ remains to be plurisubharmonic on $\cx^2$. At the same time, since $c>0$ and $\rho$ is
strictly plurisubharmonic on $S$ outside small neighbourhoods of singular points, we conclude that the function 
$\tilde \phi$ is strictly plurisubharmonic outside the balls $\mathbb B(p_j, \delta)$. It also follows 
that $X = \{ |h|=e^{\tilde\phi} \} = S$. The pair $\tilde\phi$ and $h$ now satisfies all the conditions 
of Lemma~\ref{l.g}. This completes the proof of Proposition~\ref{t.1}.

\medskip

For the proof of Theorem~\ref{t.2} we will also need  the following result.

\begin{cor}\label{c.2}
Suppose that $\iota : S \to \cx^2$ is a Lagrangian inclusion of a compact surface. Then
$\iota(S)$ admits a Stein neighbourhood basis. 
\end{cor}
Indeed, one can take neighbourhoods of $\iota(S)$ of the form $\{ \rho < \varepsilon \}$ where 
$\rho$ is a function  given by Lemma \ref{l.FS} and $\varepsilon > 0$ is small enough.

\section{Rational approximation on surfaces}

The classical Oka-Weil theorem (see, e.g.,~\cite{St}) states that any holomorphic function 
in a neighbourhood of a rationally convex compact set $X\subset \cx^n$ can be approximated 
uniformly on $X$ by rational functions with poles off $X$. Rational functions can be replaced by holomorphic 
polynomials if $X$ is polynomially convex. We will need the following approximation result, 
which is due to O'Farrel-Preskenis-Walsch \cite{FPW} (see also Stout~\cite{St}):

\medskip

\noindent{\it
Let $X$ be a compact holomorphically convex set in $\cx^n$, and let $X_0$ be a closed subset of
$X$ for which $X\setminus X_0$ is a totally real subset of the manifold $\cx^n\setminus X_0$.
A function $f\in C(X)$ can be approximated uniformly on $X$ by functions holomorphic on an 
neighbourhood of $X$ if and only if $f|_{X_0}$ can be approximated uniformly on $X_0$ by functions
holomorphic on an neighbourhood of~$X$.
}

\medskip

Recall that a set $X$ is called a {\it totally real set} of a manifold $\mathcal M$ if there is
a neighbourhood $U$ of $X$ in $\mathcal M$ on which is defined a nonnegative strictly
plurisubharmonic function $\phi$ of class $C^2$ such that $X=\{p\in U : \phi(p)=0\}$.
The following result can be found in Stout~\cite[Thm 6.2.9]{St}:

\medskip

\noindent{\it
A compact connected subset $X$ of a Stein manifold $\mathcal M$ is holomorphically convex
if and only if there is a sequence $\Omega_j$ of domains in $\mathcal M$ with $\Omega_j \supset
\Omega_k$, when $j\le k$, and with $\bigcap_j \Omega_j = X$ such that if for each $j$, 
$(\tilde \Omega_j, {\rm proj}_j)$ is the envelope of holomorphy of $\Omega_j$, then 
$\bigcap_j {\rm proj}_j(\tilde\Omega_j)=X$.
}

\medskip

Suppose now that $X=\iota(S)$ is a Lagrangian inclusion given by Proposition \ref{p.stand}; it is rationally convex by  Proposition~\ref{t.1}. 
Let $X_0$ be the set of singular points of $X$, i.e., the set of double points and
 Whitney umbrellas. Then $X\setminus X_0$ is a smooth totally real submanifold, and so
for each point $p\in X\setminus X_0$ there exists a neighbourhood in which the square of the 
distance to $X$ is a strictly plurisubharmonic function. From these neighbourhoods we can
construct a neighbourhood $U\supset X \setminus X_0$ with a nonnegative strictly plurisubharmonic 
function on it that vanishes on $X\setminus X_0$. This shows that $X\setminus X_0$ is a totally 
real set in $\cx^2 \setminus X_0$.

The set $X_0$ is finite, hence it satisfies the assumption of of
O'Farrel-Preskenis-Walsch theorem. By Lemma~\ref{c.2}, $X\subset \cx^2$ admits a Stein neighbourhood basis $\{\Omega_j\}_j$.
Each $\Omega_j$ is Stein, therefore, $\tilde \Omega_j = \Omega_j$, and it follows from above
that $X$ is holomorphically convex.  Thus, all conditions in the result of
O'Farrel-Preskenis-Walsch, stated above, are satisfied, and we conclude that any 
continuous function on $X$ can  be approximated by holomorphic functions in a neighbourhood of $X$, 
hence by rational functions as seen by the Oka-Weil theorem. Combining everything together gives the following.

\begin{prop}\label{p.cont}
If $\iota: S\to\cx^2$ is a Lagrangian inclusion with standard umbrellas, then any continuous
function on $\iota(S)$ can be approximated uniformly on $\iota(S)$ by rational functions with
poles off $\iota(S)$. 
\end{prop}

With this the main result is easily verified. 

\begin{proof}[Proof of Theorem~\ref{t.2}]
(i) and (ii). By Proposition \ref{p.stand}, there exists a singular Lagrangian embedding $f=(f_1,f_2): S \to \cx^2$
with standard umbrellas as the only singularities. The required statements now follow from Proposition~\ref{p.cont}.

(iii) Formula~\eqref{e.wh} gives an immersion of the sphere $S^2$ into $\cx^2$ with one double point, but
this does not give the approximation result because the coordinate functions attain the same 
value at the double point. However, by the Borsuk-Ulam theorem (see, e.g.,~\cite{H}),
any continuous function $F: S^2 \to \rl^2$ has at least two antipodal points $p$ and $q$ on $S^2$ 
where it attains the same value. Hence, it can be approximated by rational functions but only after 
we apply a rotation of $S^2$ that sends $p$ and $q$ to the north and south poles of $S^2$, which are 
 the preimages of the double point.

(iv) A similar story holds for $\mathbb RP_2$, for which one needs two double points. Let
$f=(f_1,f_2): \rl P_2 \to \cx^2$ be the Lagrangian inclusion with two double points and one standard umbrella.
By the Whitney approximation theorem it suffices to approximate any smooth function 
$F: \mathbb RP_2 \to \cx$. Since $\mathbb RP_2$ cannot be diffeomorphic to any subset of $\cx$, a 
generic point in the image of $F$ will have at least two pre-images. Applying a diffeomorphism $\tau$ of $\rl P_2$ 
we may assume that there exist points $p_j, q_j \in \rl P_2$ such that $(F \circ\tau)(p_j)=
(F \circ \tau)(q_j)$, $j=1,2$, and $f_j(p_k)=f_j(q_k)$, $j,k=1,2$. Then by Proposition~\ref{p.cont}, $F \circ \tau$ can be approximated by rational 
combinations of $f_1$ and $f_2$.
\end{proof}


\end{document}